\documentclass[12pt,a4paper]{amsart}
\usepackage{amssymb,amsxtra}
\usepackage[english,french]{babel}

\newtheorem{thm}{Theorem}[section]
\newtheorem{lem}[thm]{Lemma}

\theoremstyle{definition}

\theoremstyle{remark}

\theoremstyle{plain}

\theoremstyle{remark}

\newtheorem*{example}{Example}
\numberwithin{equation}{section}

\begin{document}

\title{ \textbf{Enumeration of some particular  $ 2n\times 9 $  n-times  Persymmetric  Matrices over $\mathbb{F}_{2} $ by rank}}
\author{Jorgen~Cherly}
\address{D\'epartement de Math\'ematiques, Universit\'e de
    Brest, 29238 Brest cedex~3, France}
\email{Jorgen.Cherly@univ-brest.fr}
\email{andersen69@wanadoo.fr}

\maketitle 
\begin{abstract}
Dans cet article nous comptons le nombre de certaines $2n\times 9$  n-fois matrices persym\' etriques de rang i sur $ \mathbb {F} _ {2} . $

 \end{abstract}

\selectlanguage{english}

\begin{abstract}  In this paper we count the number of some particular $2n\times 9$ n-times persymmetric rank i matrices over  $ \mathbb{F}_{2}.$
 \end{abstract}
 
   \maketitle 
\newpage
\tableofcontents
\newpage

   \section{Introduction.}
  \label{sec 1}  
  In this paper we propose to compute  the number $ \Gamma_{i}^{\left[2\atop{\vdots \atop 2}\right]\times 9}$  of rank i  $2n\times 9 $
   n-times persymmetric matrices over $\mathbb{F}_{2}$ of the below form for $0\leqslant i\leqslant \inf(2n,9) $  \\

    \begin{equation}
  \label{eq 3.1}
   \left (  \begin{array} {ccccccccc}
\alpha  _{1}^{(1)} & \alpha  _{2}^{(1)}  &   \alpha_{3}^{(1)} &   \alpha_{4}^{(1)} &   \alpha_{5}^{(1)} &  \alpha_{6}^{(1)}  & \alpha_{7}^{(1)}  & \alpha_{8}^{(1)} & \alpha_{9}^{(1)} \\
\alpha  _{2}^{(1)} & \alpha  _{3}^{(1)}  &   \alpha_{4}^{(1)} &   \alpha_{5}^{(1)} &   \alpha_{6}^{(1)} &  \alpha_{7}^{(1)} & \alpha_{8}^{(1)}  &  \alpha_{9}^{(1)}& \alpha_{10}^{(1)} \\ 
\hline \\
\alpha  _{1}^{(2)} & \alpha  _{2}^{(2)}  &   \alpha_{3}^{(2)} &   \alpha_{4}^{(2)} &   \alpha_{5}^{(2)} &  \alpha_{6}^{(2)} &  \alpha_{7}^{(2)}  & \alpha_{8}^{(2)} & \alpha_{9}^{(2)} \\
\alpha  _{2}^{(2)} & \alpha  _{3}^{(2)}  &   \alpha_{4}^{(2)} &   \alpha_{5}^{(2)}&   \alpha_{6}^{(2)} &  \alpha_{7}^{(2)}  &  \alpha_{8}^{(2)} &  \alpha_{9}^{(2)} &\alpha_{10}^{(2)} \\ 
\hline\\
\alpha  _{1}^{(3)} & \alpha  _{2}^{(3)}  &   \alpha_{3}^{(3)}  &   \alpha_{4}^{(3)} &   \alpha_{5}^{(3)} &  \alpha_{6}^{(3)} &  \alpha_{7}^{(3)} &  \alpha_{8}^{(3)} &\alpha_{9}^{(3)} \\
\alpha  _{2}^{(3)} & \alpha  _{3}^{(3)}  &   \alpha_{4}^{(3)}&   \alpha_{5}^{(3)} &   \alpha_{6}^{(3)}  &  \alpha_{7}^{(3)} & \alpha_{8}^{(3)}  &   \alpha_{9}^{(3)}&\alpha_{10}^{(3)} \\ 
\hline \\
\vdots & \vdots & \vdots  & \vdots  & \vdots & \vdots  & \vdots & \vdots  & \vdots \\
\hline \\
\alpha  _{1}^{(n)} & \alpha  _{2}^{(n)}  &   \alpha_{3}^{(n)} &   \alpha_{4}^{(n)} &   \alpha_{5}^{(n)}  &  \alpha_{6}^{(n)} &  \alpha_{7}^{(n)}  &  \alpha_{8}^{(n)}& \alpha_{9}^{(n)} \\
\alpha  _{2}^{(n)} & \alpha  _{3}^{(n)}  &   \alpha_{4}^{(n)}&   \alpha_{5}^{(n)} &   \alpha_{6}^{(n)}  &  \alpha_{7}^{(n)} &  \alpha_{8}^{(n)}  &   \alpha_{9}^{(n)}&\alpha_{10}^{(n)} \\ 
\end{array} \right )  
\end{equation} 
We remark that this paper is based on the results in  the author's papers [11,12].

   \section{Some notations concerning the field of Laurent Series $ \mathbb{F}_{2}((T^{-1})) $. }
  \label{sec 2}  
  We denote by $ \mathbb{F}_{2}\big(\big({T^{-1}}\big) \big)
 = \mathbb{K} $ the completion
 of the field $\mathbb{F}_{2}(T), $  the field of  rational fonctions over the
 finite field\; $\mathbb{F}_{2}$,\; for the  infinity  valuation \;
 $ \mathfrak{v}=\mathfrak{v}_{\infty }$ \;defined by \;
 $ \mathfrak{v}\big(\frac{A}{B}\big) = degB -degA $ \;
 for each pair (A,B) of non-zero polynomials.
 Then every element non-zero t in
  $\mathbb{F}_{2}\big(\big({\frac{1}{T}}\big) \big) $
 can be expanded in a unique way in a convergent Laurent series
                              $  t = \sum_{j= -\infty }^{-\mathfrak{v}(t)}t_{j}T^j
                                 \; where\; t_{j}\in \mathbb{F}_{2}. $\\
  We associate to the infinity valuation\; $\mathfrak{v}= \mathfrak{v}_{\infty }$
   the absolute value \; $\vert \cdot \vert_{\infty} $\; defined by \;
  \begin{equation*}
  \vert t \vert_{\infty} =  \vert t \vert = 2^{-\mathfrak{v}(t)}. \\
\end{equation*}
    We denote  E the  Character of the additive locally compact group
$  \mathbb{F}_{2}\big(\big({\frac{1}{T}}\big) \big) $ defined by \\
\begin{equation*}
 E\big( \sum_{j= -\infty }^{-\mathfrak{v}(t)}t_{j}T^j\big)= \begin{cases}
 1 & \text{if      }   t_{-1}= 0, \\
  -1 & \text{if      }   t_{-1}= 1.
    \end{cases}
\end{equation*}
  We denote $\mathbb{P}$ the valuation ideal in $ \mathbb{K},$ also denoted the unit interval of  $\mathbb{K},$ i.e.
  the open ball of radius 1 about 0 or, alternatively, the set of all Laurent series 
   $$ \sum_{i\geq 1}\alpha _{i}T^{-i}\quad (\alpha _{i}\in  \mathbb{F}_{2} ) $$ and, for every rational
    integer j,  we denote by $\mathbb{P}_{j} $
     the  ideal $\left\{t \in \mathbb{K}|\; \mathfrak{v}(t) > j \right\}. $
     The sets\; $ \mathbb{P}_{j}$\; are compact subgroups  of the additive
     locally compact group \; $ \mathbb{K}. $\\
      All $ t \in \mathbb{F}_{2}\Big(\Big(\frac{1}{T}\Big)\Big) $ may be written in a unique way as
$ t = [t] + \left\{t\right\}, $ \;  $  [t] \in \mathbb{F}_{2}[T] ,
 \; \left\{t\right\}\in \mathbb{P}  ( =\mathbb{P}_{0}). $\\
 We denote by dt the Haar measure on \; $ \mathbb{K} $\; chosen so that \\
  $$ \int_{\mathbb{P}}dt = 1. $$\\
  
  $$ Let \quad
  (t_{1},t_{2},\ldots,t_{n} )
 =  \big( \sum_{j=-\infty}^{-\nu(t_{1})}\alpha _{j}^{(1)}T^{j},  \sum_{j=-\infty}^{-\nu(t_{2})}\alpha _{j}^{(2)}T^{j} ,\ldots, \sum_{j=-\infty}^{-\nu(t_{n})}\alpha _{j}^{(n)}T^{j}\big) \in  \mathbb{K}^{n}. $$ 
 We denote $\psi  $  the  Character on  $(\mathbb{K}^n, +) $ defined by \\
 \begin{align*}
  \psi \big( \sum_{j=-\infty}^{-\nu(t_{1})}\alpha _{j}^{(1)}T^{j},  \sum_{j=-\infty}^{-\nu(t_{2})}\alpha _{j}^{(2)}T^{j} ,\ldots, \sum_{j=-\infty}^{-\nu(t_{n})}\alpha _{j}^{(n)}T^{j}\big) & = E \big( \sum_{j=-\infty}^{-\nu(t_{1})}\alpha _{j}^{(1)}T^{j}\big) \cdot E\big( \sum_{j=-\infty}^{-\nu(t_{2})}\alpha _{j}^{(2)}T^{j}\big)\cdots E\big(  \sum_{j=-\infty}^{-\nu(t_{n})}\alpha _{j}^{(n)}T^{j}\big) \\
  & = 
    \begin{cases}
 1 & \text{if      }     \alpha _{-1}^{(1)} +    \alpha _{-1}^{(2)}  + \ldots +   \alpha _{-1}^{(n)}   = 0 \\
  -1 & \text{if      }    \alpha _{-1}^{(1)} +    \alpha _{-1}^{(2)}  + \ldots +   \alpha _{-1}^{(n)}   =1                                                                                                                          
    \end{cases}
  \end{align*}
   \section{Some results concerning  n-times persymmetric matrices over  $ \mathbb{F}_{2}$.}
  \label{sec 3}  
     $$ Set\quad
  (t_{1},t_{2},\ldots,t_{n} )
 =  \big( \sum_{i\geq 1}\alpha _{i}^{(1)}T^{-i}, \sum_{i \geq 1}\alpha  _{i}^{(2)}T^{-i},\sum_{i \geq 1}\alpha _{i}^{(3)}T^{-i},\ldots,\sum_{i \geq 1}\alpha _{i}^{(n)}T^{-i}   \big) \in  \mathbb{P}^{n}. $$

     Denote by $D^{\left[2 \atop{\vdots \atop 2}\right]\times k}(t_{1},t_{2},\ldots,t_{n} ) $
    
    the following $2n \times k $ \;  n-times  persymmetric  matrix  over the finite field  $\mathbb{F}_{2}. $ 
    
  \begin{equation}
  \label{eq 3.1}
   \left (  \begin{array} {cccccccc}
\alpha  _{1}^{(1)} & \alpha  _{2}^{(1)}  &   \alpha_{3}^{(1)} &   \alpha_{4}^{(1)} &   \alpha_{5}^{(1)} &  \alpha_{6}^{(1)}  & \ldots  &  \alpha_{k}^{(1)} \\
\alpha  _{2}^{(1)} & \alpha  _{3}^{(1)}  &   \alpha_{4}^{(1)} &   \alpha_{5}^{(1)} &   \alpha_{6}^{(1)} &  \alpha_{7}^{(1)} & \ldots  &  \alpha_{k+1}^{(1)} \\ 
\hline \\
\alpha  _{1}^{(2)} & \alpha  _{2}^{(2)}  &   \alpha_{3}^{(2)} &   \alpha_{4}^{(2)} &   \alpha_{5}^{(2)} &  \alpha_{6}^{(2)} & \ldots   &  \alpha_{k}^{(2)} \\
\alpha  _{2}^{(2)} & \alpha  _{3}^{(2)}  &   \alpha_{4}^{(2)} &   \alpha_{5}^{(2)}&   \alpha_{6}^{(2)} &  \alpha_{7}^{(2)}  & \ldots  &  \alpha_{k+1}^{(2)} \\ 
\hline\\
\alpha  _{1}^{(3)} & \alpha  _{2}^{(3)}  &   \alpha_{3}^{(3)}  &   \alpha_{4}^{(3)} &   \alpha_{5}^{(3)} &  \alpha_{6}^{(3)} & \ldots  &  \alpha_{k}^{(3)} \\
\alpha  _{2}^{(3)} & \alpha  _{3}^{(3)}  &   \alpha_{4}^{(3)}&   \alpha_{5}^{(3)} &   \alpha_{6}^{(3)}  &  \alpha_{7}^{(3)} & \ldots  &  \alpha_{k+1}^{(3)} \\ 
\hline \\
\vdots & \vdots & \vdots  & \vdots  & \vdots & \vdots  & \vdots & \vdots \\
\hline \\
\alpha  _{1}^{(n)} & \alpha  _{2}^{(n)}  &   \alpha_{3}^{(n)} &   \alpha_{4}^{(n)} &   \alpha_{5}^{(n)}  &  \alpha_{6}^{(n)} & \ldots  &  \alpha_{k}^{(n)} \\
\alpha  _{2}^{(n)} & \alpha  _{3}^{(n)}  &   \alpha_{4}^{(n)}&   \alpha_{5}^{(n)} &   \alpha_{6}^{(n)}  &  \alpha_{7}^{(n)} & \ldots  &  \alpha_{k+1}^{(n)} \\ 
\end{array} \right )  
\end{equation} 
We denote by  $ \Gamma_{i}^{\left[2\atop{\vdots \atop 2}\right]\times k}$  the number of rank i n-times persymmetric matrices over $\mathbb{F}_{2}$ of the above form :  \\

  Let $ \displaystyle  f (t_{1},t_{2},\ldots,t_{n} ) $  be the exponential sum  in $ \mathbb{P}^{n} $ defined by\\
    $(t_{1},t_{2},\ldots,t_{n} ) \displaystyle\in \mathbb{P}^{n}\longrightarrow \\
    \sum_{deg Y\leq k-1}\sum_{deg U_{1}\leq  1}E(t_{1} YU_{1})
  \sum_{deg U_{2} \leq 1}E(t _{2} YU_{2}) \ldots \sum_{deg U_{n} \leq 1} E(t _{n} YU_{n}). $\vspace{0.5 cm}\\
    Then
  $$     f_{k} (t_{1},t_{2},\ldots,t_{n} ) =
  2^{2n+k- rank\big[ D^{\left[2\atop{\vdots \atop 2}\right]\times k}(t_{1},t_{2},\ldots,t_{n} )\big] } $$

    Hence  the number denoted by $ R_{q,n}^{(k)} $ of solutions \\
  
 $(Y_1,U_{1}^{(1)},U_{2}^{(1)}, \ldots,U_{n}^{(1)}, Y_2,U_{1}^{(2)},U_{2}^{(2)}, 
\ldots,U_{n}^{(2)},\ldots  Y_q,U_{1}^{(q)},U_{2}^{(q)}, \ldots,U_{n}^{(q)}   ) \in (\mathbb{F}_{2}[T])^{(n+1)q}$ \vspace{0.5 cm}\\
 of the polynomial equations  \vspace{0.5 cm}
  \[\left\{\begin{array}{c}
 Y_{1}U_{1}^{(1)} + Y_{2}U_{1}^{(2)} + \ldots  + Y_{q}U_{1}^{(q)} = 0  \\
    Y_{1}U_{2}^{(1)} + Y_{2}U_{2}^{(2)} + \ldots  + Y_{q}U_{2}^{(q)} = 0\\
    \vdots \\
   Y_{1}U_{n}^{(1)} + Y_{2}U_{n}^{(2)} + \ldots  + Y_{q}U_{n}^{(q)} = 0 
 \end{array}\right.\]
 
    $ \Leftrightarrow
    \begin{pmatrix}
   U_{1}^{(1)} & U_{1}^{(2)} & \ldots  & U_{1}^{(q)} \\ 
      U_{2}^{(1)} & U_{2}^{(2)}  & \ldots  & U_{2}^{(q)}  \\
\vdots &   \vdots & \vdots &   \vdots   \\
U_{n}^{(1)} & U_{n}^{(2)}   & \ldots  & U_{n}^{(q)} \\
 \end{pmatrix}  \begin{pmatrix}
   Y_{1} \\
   Y_{2}\\
   \vdots \\
   Y_{q} \\
  \end{pmatrix} =   \begin{pmatrix}
  0 \\
  0 \\
  \vdots \\
  0 
  \end{pmatrix} $\\

    satisfying the degree conditions \\
                   $$  degY_i \leq k-1 ,
                   \quad degU_{j}^{(i)} \leq 1, \quad  for \quad 1\leq j\leq n,  \quad 1\leq i \leq q $$ \\
  is equal to the following integral over the unit interval in $ \mathbb{K}^{n} $
    $$ \int_{\mathbb{P}^{n}} f_{k}^{q}(t_{1},t_{2},\ldots,t_{n}) dt_{1}dt _{2}\ldots dt _{n}. $$
  Observing that $ f (t_{1},t_{2},\ldots,t_{n} ) $ is constant on cosets of $ \prod_{j=1}^{n}\mathbb{P}_{k+1} $ in $ \mathbb{P}^{n} $\;
  the above integral is equal to 
  
  \begin{equation}
  \label{eq 3.2}
 2^{q(2n+k) - (k+1)n}\sum_{i = 0}^{\inf(2n,k)}
  \Gamma_{i}^{\left[2\atop{\vdots \atop 2}\right]\times k} 2^{-iq} =  R_{q,n}^{(k)}. 
 \end{equation}
 
 \begin{eqnarray}
 \label{eq 3.3}
\text{ Recall that $ R_{q,n}^{(k)}$ is equal to the number of solutions of the polynomial system} \nonumber \\
    \begin{pmatrix}
   U_{1}^{(1)} & U_{1}^{(2)} & \ldots  & U_{1}^{(q)} \\ 
      U_{2}^{(1)} & U_{2}^{(2)}  & \ldots  & U_{2}^{(q)}  \\
\vdots &   \vdots & \vdots &   \vdots   \\
U_{n}^{(1)} & U_{n}^{(2)}   & \ldots  & U_{n}^{(q)} \\
 \end{pmatrix}  \begin{pmatrix}
   Y_{1} \\
   Y_{2}\\
   \vdots \\
   Y_{q} \\
  \end{pmatrix} =   \begin{pmatrix}
  0 \\
  0 \\
  \vdots \\
  0 
  \end{pmatrix} \\
 \text{ satisfying the degree conditions}\nonumber \\
                     degY_i \leq k-1 ,
                   \quad degU_{j}^{(i)} \leq 1, \quad  for \quad 1\leq j\leq n,   \quad 1\leq i \leq q. \nonumber
 \end{eqnarray}

 From \eqref{eq 3.2} we obtain for q = 1\\
   \begin{align}
  \label{eq 3.4}
 2^{k-(k-1)n}\sum_{i = 0}^{\inf(2n,k)}
 \Gamma_{i}^{\left[2\atop{\vdots \atop 2}\right]\times k} 2^{-i} =  R_{1,n}^{(k)} = 2^{2n}+2^k-1.
  \end{align}

We have obviously \\

   \begin{align}
  \label{eq 3.5}
 \sum_{i = 0}^{k}
 \Gamma_{i}^{\left[2\atop{\vdots \atop 2}\right]\times k}  = 2^{(k+1)n}.  
 \end{align}

From  the fact that the number of rank one persymmetric  matrices over $\mathbb{F}_{2}$ is equal to three  we obtain using
 combinatorial methods  : \\
 
    \begin{equation}
  \label{eq 3.6}
 \Gamma_{1}^{\left[2\atop{\vdots \atop 2}\right]\times k}  = (2^{n}-1)\cdot 3.
  \end{equation}
  For more details see Cherly  [11]
  
 \subsection{Computation of $ \Gamma_{6}^{\left[2\atop{\vdots \atop 2}\right]\times k} $.}
We recall (see section \ref{sec 3} ) that $ \Gamma_{6}^{\left[2\atop{\vdots \atop 2}\right]\times k}$ denotes the number of rank 6
n-times persymmetric matrices over  $\mathbf{F}_2 $  of the form \eqref{eq 3.1}
We shall need the following Lemma  : \\
  \begin{lem}
\label{lem 3.1}
\begin{equation}
\label{eq 3.7}
   \Gamma_{6}^{\left[2\atop{\vdots \atop 2}\right]\times k}=   \begin{cases}
 0 & \text{if  } n = 0,  \\  
0 & \text{if  } n = 1,   \\
0 & \text{if  } n = 2,\\
2^{3k+3}-7\cdot 2^{2k+6}+7\cdot 2^{k+10}-32768 & \text{if   } n=3,\\
120\cdot 2^{3k}+123480\cdot 2^{2k}-6142080\cdot 2^{k}+66170880  & \text{if   }  n = 4, \\
1240\cdot[ 2^{3k}+3199\cdot2^{2k}+2^7\cdot3913\cdot2^{k}-18883\cdot2^{10}] & \text{if   }  n = 5.
 \end{cases}
   \end{equation} 
   
    \begin{equation}
  \label{eq 3.8}
    \Gamma_{6}^{\left[2\atop{\vdots \atop 2}\right]\times 7}=   127\cdot 2^{6n}-189\cdot2^{5n} -7378\cdot 2^{4n}\\
    +24240\cdot 2^{3n}+35168 \cdot 2^{2n}-166656 \cdot2^{n} +114688. 
      \end{equation}
\end{lem}
\begin{proof}
Lemma \ref{lem 3.1} follows from  Cherly[11,12].
\end{proof}

  \begin{lem}
  \label{lem 3.2}
We postulate that :\\
\begin{align}
\label{eq 3.9}
  \displaystyle  \Gamma_{6}^{\left[2\atop{\vdots \atop 2}\right]\times k} =   \displaystyle 127 \cdot 2^{6n} +[ 651\cdot 2^{k-3}-10605 ] \cdot 2^{5n}+ [\frac{155}{3}\cdot 2^{2k-3} -22661\cdot 2^{k-3}  +\frac{748154}{3}] \cdot 2^{4n} \\
  +  \frac{1}{168}\cdot [2^{3k+3}-16723\cdot2^{2k}+5026378 \cdot 2^{k}-382091648 ]   \cdot 2^{3n} \nonumber \\
   +[ - \frac{1}{3} \cdot 2^{3k} + \frac{5649}{12}\cdot 2^{2k} -\frac{368711}{3}\cdot 2^{k} +8753120 ] \cdot 2^{2n} \nonumber \\
   +[\frac{2}{3} \cdot 2^{3k} - \frac{2437}{3}\cdot 2^{2k} + \frac{597736}{3}\cdot 2^{k} -\frac{41276672}{3}] \cdot 2^{n} \nonumber \\
    -8\cdot [ \frac{1}{21}\cdot 2^{3k}-\frac{163}{3} \cdot 2^{2k}+\frac{38816}{3} \cdot 2^k-\frac{18483200}{21} ] \quad \text{for} \quad k\geqslant 7.  \nonumber
 \end{align}
\end{lem}

 \begin{proof}
 \begin{align*}
   \text{ From the expression of $ \Gamma_{6}^{\left[2\atop{\vdots \atop 2}\right]\times k} $  in  \eqref{eq 3.8} for k=7  we assume that }\\
 \Gamma_{6}^{\left[2\atop{\vdots \atop 2}\right]\times k} \quad \text{can be written in the form} :\\
   \displaystyle 127 \cdot 2^{6n} + a(k) \cdot 2^{5n}+ b(k) \cdot 2^{4n}+c(k) \cdot 2^{3n} 
 +d(k) \cdot 2^{2n} +e(k) \cdot 2^{n} +f(k)  \\
\text{Set} \;  Y=2^n, \; \text{then}\; \Gamma_{6}^{\left[2\atop{\vdots \atop 2}\right]\times k} =   127\cdot Y^6  +a(k) \cdot Y^5  +b(k) \cdot Y^4+ c(k) \cdot Y^3 + d(k)\cdot Y^2+e(k) \cdot Y +f(k) \\
\Gamma_{6}^{\left[2\atop{\vdots \atop 2}\right]\times k} = 0 \quad  \text{ for}\; n\in\{0,1,2\}\\
\text{Then} \; \Gamma_{6}^{\left[2\atop{\vdots \atop 2}\right]\times k} = (Y-1)(Y-2)(Y-4)[127 \cdot Y^3+\alpha(k)\cdot Y^2+\beta(k) \cdot Y+\gamma(k)] \\
 = [2^{3n}-7\cdot 2^{2n}+14\cdot 2^n-8] \cdot [ 127 \cdot 2^{3n}+\alpha(k) \cdot 2^{2n}+\beta(k) \cdot 2^{n}+\gamma(k)] \\
= 127\cdot 2^{3n}+2^{5n} \cdot [\alpha(k)-7\cdot 127]+2^{4n}\cdot [\beta(k)-7\cdot \alpha(k)+14\cdot 127] \\
+2^{3n}\cdot [\gamma(k)-7\cdot \beta(k)+14\cdot \alpha(k)-8 \cdot 127] +2^{2n}\cdot [-7\cdot \gamma(k)+14\cdot \beta(k)-8 \cdot \alpha(k) ]\\
+2^{n} \cdot [14\cdot \gamma(k)-8\cdot \beta(k)] -8\cdot \gamma(k)\\
 \displaystyle  = 127 \cdot 2^{6n} + a(k) \cdot 2^{5n}+ b(k) \cdot 2^{4n}+c(k) \cdot 2^{3n} 
 +d(k) \cdot 2^{2n} +e(k) \cdot 2^{n} +f(k).  \\
 \end{align*}
 The case n=3 : \\
 \begin{align*}
\Gamma_{6}^{\left[2\atop{\vdots \atop 2}\right]\times k} = (2^3-1)(2^3-2)(2^3-4)[127 \cdot 2^9+\alpha(k)\cdot 2^6+\beta(k) \cdot 2^3+\gamma(k)] \\
= 168\cdot [127 \cdot 2^9+\alpha(k)\cdot 2^6+\beta(k) \cdot 2^3+\gamma(k)] = 2^{3k+3}-7\cdot2^{2k+6}+7\cdot2^{k+10}-32768 \\
 \Rightarrow 127 \cdot 2^9+\alpha(k)\cdot 2^6+\beta(k) \cdot 2^3+\gamma(k) = \frac{1}{168}\cdot[ 2^{3k+3}-7\cdot2^{2k+6}+7\cdot2^{k+10}-32768 ]\\  \Rightarrow     (1) \quad       64\cdot \alpha(k)+8\cdot \beta(k)+\gamma (k)=\frac{1}{21}\cdot[ 2^{3k}-56\cdot2^{2k}+896\cdot2^{k}-2^{12} ] -127 \cdot 2^9. 
 \end{align*}
 The case n=4 :\\
 \begin{align*}
\Gamma_{6}^{\left[2\atop{\vdots \atop 2}\right]\times k} = (2^4-1)(2^4-2)(2^4-4)[127 \cdot 2^{12}+\alpha(k)\cdot 2^8+\beta(k) \cdot 2^4+\gamma(k)]\\
= 2520\cdot [127 \cdot 2^{12}+\alpha(k)\cdot 2^8+\beta(k) \cdot 2^4+\gamma(k)]\\ =120\cdot 2^{3k}+123480\cdot2^{2k}-6142080\cdot2^{k+10}+66170880 \\
\Rightarrow 127 \cdot 2^{12}+\alpha(k)\cdot 2^8+\beta(k) \cdot 2^4+\gamma(k) \\ = \frac{1}{2520}\cdot[ 120\cdot 2^{3k}+123480\cdot2^{2k}-6142080\cdot2^{k+10}+66170880 ] \\
    \Rightarrow     (2) \quad       256\cdot \alpha(k)+16\cdot \beta(k)+\gamma (k)=\frac{1}{21}\cdot[ 2^{3k}+1029\cdot2^{2k}-51184\cdot 2^{k}+551424 ] -127 \cdot 2^{12}. 
    \end{align*}
     The case n=5 :\\
 \begin{align*}
\Gamma_{6}^{\left[2\atop{\vdots \atop 2}\right]\times k} = (2^5-1)(2^5-2)(2^5-4)[127 \cdot 2^{15}+\alpha(k)\cdot 2^{10}+\beta(k) \cdot 2^5+\gamma(k)] \\
= 26040\cdot [127 \cdot 2^{15}+\alpha(k)\cdot 2^{10}+\beta(k) \cdot 2^5+\gamma(k)] \\ =1240\cdot[ 2^{3k}+3199\cdot2^{2k}+2^7\cdot3913\cdot2^{k}-18883\cdot2^{10} \\
\Rightarrow 127 \cdot 2^{15}+\alpha(k)\cdot 2^{10}+\beta(k) \cdot 2^5+\gamma(k)\\ = \frac{1240}{26040}\cdot[ 2^{3k}+3199\cdot2^{2k}+2^7\cdot3913\cdot2^{k}-18883\cdot2^{10} ]  \\
     \Rightarrow     (3) \quad       1024\cdot \alpha(k)+32\cdot \beta(k)+\gamma (k)=\frac{1}{21}\cdot[ 2^{3k}+3199\cdot2^{2k}+2^7\cdot3913\cdot 2^{k}-18883\cdot2^{10} ] -127 \cdot 2^{15}. \\
\text{From (1),(2),(3) we obtain:}\begin{cases}
\alpha(k) = 651\cdot 2^{k-3}-2429 \cdot 2^2. \\
 \beta(k)= \frac{155}{3}\cdot2^{2k-3}-2263\cdot2^{k}+\frac{538784}{3}. \\
 \gamma (k)= \frac{1}{21}\cdot 2^{3k}-\frac{163}{3} \cdot 2^{2k}+\frac{38816}{3} \cdot 2^k-\frac{18483200}{21}. 
 \end{cases}
 \end{align*}
 
  \text{We then deduce :}
  \begin{equation*}
  \begin{cases} 
 \displaystyle a(k) = \alpha(k)-7\cdot 127= 651\cdot 2^{k-3}-10605 \\
  \displaystyle  b(k) = \beta(k)-7\cdot \alpha(k)+14\cdot 127= \frac{155}{3}\cdot 2^{2k-3} -22661\cdot 2^{k-3}  +\frac{748154}{3}\\
  \displaystyle c(k) = \gamma(k)-7\cdot \beta(k)+14\cdot \alpha(k)-8 \cdot 127 = \\
  \frac{1}{168}\cdot [2^{3k+3}-16723\cdot2^{2k}+5026378 \cdot 2^{k}-382091648 ] \\
  \displaystyle d(k) = -7\cdot \gamma(k)+14\cdot \beta(k)-8 \cdot \alpha(k) = -\frac{1}{3} \cdot 2^{3k} + \frac{5649}{12}\cdot 2^{2k} -
  \frac{368711}{3}\cdot 2^{k} +8753120 \\
  \displaystyle e(k) = 14\cdot \gamma(k)-8\cdot \beta(k) = \frac{2}{3} \cdot 2^{3k} - \frac{2437}{3}\cdot 2^{2k} +
  \frac{597736}{3}\cdot 2^{k} -\frac{41276672}{3} \\
  \displaystyle f(k) = -8 \cdot \gamma(k) = -8\cdot [ \frac{1}{21}\cdot 2^{3k}-\frac{163}{3} \cdot 2^{2k}+\frac{38816}{3} \cdot 2^k-\frac{18483200}{21} ]
\end{cases}
    \end{equation*}
 \end{proof}
  \subsection{Computation of $ \Gamma_{i}^{\left[2\atop{\vdots \atop 2}\right]\times 9} for \; 0\leqslant i\leqslant \inf(2n,9)$.}
 We shall need the following Lemma : \\
  \begin{lem}
\label{lem 3.3}
 \begin{equation}
  \label{eq 3.10}
 \begin{cases} 
 \displaystyle  \Gamma_{0}^{\left[2\atop{\vdots \atop 2}\right]\times k}  = 1 \quad \text{if} \quad  k\geqslant 1 \\
\displaystyle  \Gamma_{1}^{\left[2\atop{\vdots \atop 2}\right]\times k}  = (2^{n}-1)\cdot 3 \quad \text{if} \quad  k\geqslant 2 \\
\displaystyle \Gamma_{2}^{\left[2\atop{\vdots \atop 2}\right]\times k} = 7\cdot2^{2n}+(2^{k+1}-25) \cdot 2^{n}-2^{k+1}+18 \quad \text{for} \quad k\geqslant 3\\
\displaystyle  \Gamma_{3}^{\left[2\atop{\vdots \atop 2}\right]\times k} = 15\cdot2^{3n} + (7\cdot2^k-133)\cdot2^{2n}+ (294-21\cdot 2^k) \cdot 2^{n}   -176+14\cdot2^k \quad \text{for} \quad k\geqslant 4\\
\displaystyle  \Gamma_{4}^{\left[2\atop{\vdots \atop 2}\right]\times k} = 31\cdot2^{4n} + \frac{35\cdot2^{k}-1210}{2}\cdot2^{3n}
+ \frac{2^{2k+2}-783\cdot2^{k}+19028}{6}\cdot 2^{2n}\\
\displaystyle +(-2^{2k+1}+269\cdot2^{k}-5744)\cdot2^n 
 +\frac{2^{2k+2}-117\cdot2^{k+2}+9440}{3}  \quad \text{for} \quad k\geqslant 5\\
  \displaystyle  \Gamma_{5}^{\left[2\atop{\vdots \atop 2}\right]\times k} = 63\cdot2^{5n} + (\frac{155}{4}\cdot2^{k}-2573)\cdot2^{4n}
+ (\frac{5}{2}\cdot2^{2k}-\frac{2565}{4}\cdot2^{k}+29150)\cdot2^{3n}\\
\displaystyle +\frac{1}{2}\cdot(-35\cdot2^{2k}+6265\cdot2^{k}-247520)\cdot2^{2n} 
\displaystyle +(35\cdot2^{2k}-5490\cdot2^{k}+203872)\cdot2^{n}\\
-20\cdot2^{2k}+2960\cdot2^{k}-106752  \quad \text{for} \quad k\geqslant 6\\
 \displaystyle  \Gamma_{6}^{\left[2\atop{\vdots \atop 2}\right]\times k} =   127 \cdot 2^{6n} +[ 651\cdot 2^{k-3}-10605 ] \cdot 2^{5n}+ [\frac{155}{3}\cdot 2^{2k-3} -22661\cdot 2^{k-3}  +\frac{748154}{3}] \cdot 2^{4n} \\
  +  \frac{1}{168}\cdot [2^{3k+3}-16723\cdot2^{2k}+5026378 \cdot 2^{k}-382091648 ]   \cdot 2^{3n} \\
   +[ - \frac{1}{3} \cdot 2^{3k} + \frac{5649}{12}\cdot 2^{2k} -\frac{368711}{3}\cdot 2^{k} +8753120 ] \cdot 2^{2n} \\
   +[\frac{2}{3} \cdot 2^{3k} - \frac{2437}{3}\cdot 2^{2k} + \frac{597736}{3}\cdot 2^{k} -\frac{41276672}{3}] \cdot 2^{n}\\
    -8\cdot [ \frac{1}{21}\cdot 2^{3k}-\frac{163}{3} \cdot 2^{2k}+\frac{38816}{3} \cdot 2^k-\frac{18483200}{21} ]    \quad \text{for} \quad k\geqslant 7\\
 \end{cases}
 \end{equation}
 \begin{equation}
 \label{eq 3.11}
  \begin{cases}  
\displaystyle  \sum_{i = 0}^{\inf (2n,k)} \Gamma_{i}^{\left[2\atop{\vdots \atop 2}\right]\times k}  = 2^{(k+1)n}, \\ 
  \displaystyle  \sum_{i = 0}^{\inf (2n,k)} \Gamma_{i}^{\left[2\atop{\vdots \atop 2}\right]\times k} 2^{-i}  = 2^{n+k(n-1)}+2^{(k-1)n}-2^{(k-1)n-k},\\
  \displaystyle \sum_{i = 0}^{\inf (2n,k)} \Gamma_{i}^{\left[2\atop{\vdots \atop 2}\right]\times k} 2^{-2i}  =
   2^{n+k(n-2)}+2^{-n+k(n-2)}\cdot[3\cdot2^k-3] +2^{-2n+k(n-2)}\cdot[6\cdot2^{k-1}-6] \\
   +2^{-3n+kn}-6\cdot2^{n(k-3)-k}+8\cdot2^{-3n+k(n-2)}.
\end{cases}
    \end{equation}
  \end{lem}  
 \begin{proof}
 Lemma 3.3 follows from Lemma 6.3 in Cherly [12] and \eqref{eq 3.9}
 \end{proof}
 We deduce from \eqref{eq 3.10} and \eqref{eq 3.11}  with k=9.
 
    \begin{equation}
  \label{eq 3.12}
  \Gamma_{i}^{\left[2\atop{\vdots \atop 2}\right]\times 9} = \begin{cases}
1 & \text{if  } i = 0,        \\
 (2^{n}-1)\cdot 3 & \text{if   } i=1,\\
 7\cdot2^{2n}+999 \cdot 2^{n}-1006  & \text{if   }  i = 2,  \\
 15\cdot 2^{3n}+3451\cdot 2^{2n}
-10458\cdot 2^{n}+6992  & \text{if   }  i = 3,  \\
 31\cdot 2^{4n} +8355\cdot 2^{3n}+111118 \cdot 2^{2n}
-392304 \cdot2^{n}+272800 & \text{if   }  i=4, \\ 
 63 \cdot 2^{5n}+17267 \cdot2^{4n} +356190 \cdot2^{3n}
 -3107440 \cdot 2^{2n}+6568032 \cdot 2^n -3834112 & \text{if   }  i=5, \\
 127\cdot 2^{6n}+31059\cdot2^{5n}+ 492094\cdot 2^{4n}  -6658800\cdot 2^{3n}\\
 +24491488 \cdot 2^{2n}-35215104\cdot2^{n} +16859136 & \text{if   }  i=6, \\
  255\cdot 2^{7n}  +a_{6}^{(7)}\cdot2^{6n}  +a_{5}^{(7)}\cdot2^{5n}+ a_{4}^{(7)}\cdot 2^{4n}\\ +
 a_{3}^{(7)}\cdot 2^{3n}+a_{2}^{(7)} \cdot 2^{2n}+ a_{1}^{(7)}\cdot2^{n} +a_{0}^{(7)}& \text{if   }  i=7, \\
 511\cdot2^{8n} +a_{7}^{(8)}\cdot2^{7n}  +a_{6}^{(8)}\cdot2^{6n}  +a_{5}^{(8)}\cdot2^{5n}+ a_{4}^{(8)}\cdot 2^{4n} \\ +
 a_{3}^{(8)}\cdot 2^{3n}+a_{2}^{(8)} \cdot 2^{2n}+ a_{1}^{(8)}\cdot2^{n} +a_{0}^{(8)}& \text{if   }  i=8, \\
 2^{10n} +a_{8}^{(9)}\cdot2^{8n} +a_{7}^{(9)}\cdot2^{7n}  +a_{6}^{(9)}\cdot2^{6n}  +a_{5}^{(9)}\cdot2^{5n}+ a_{4}^{(9)}\cdot 2^{4n} \\ +
 a_{3}^{(9)}\cdot 2^{3n}+a_{2}^{(9)} \cdot 2^{2n}+ a_{1}^{(9)}\cdot2^{n} +a_{0}^{(9)}& \text{if   }  i=9. \\
 \end{cases}    
  \end{equation}
where, \\
 \begin{equation}
  \label{eq 3.13}
 \begin{cases} 
 \displaystyle  \sum_{i = 0}^{9} \Gamma_{i}^{\left[2\atop{\vdots \atop 2}\right]\times 9}  = 2^{10n}, \\ 
  \displaystyle  \sum_{i = 0}^{9} \Gamma_{i}^{\left[2\atop{\vdots \atop 2}\right]\times 9} 2^{9-i}  = 2^{10n}+511\cdot 2^{8n},\\
  \displaystyle \sum_{i = 0}^{9} \Gamma_{i}^{\left[2\atop{\vdots \atop 2}\right]\times 9} 2^{18-2i}  =
   2^{10n}+1533\cdot2^{8n} +1530\cdot2^{7n} +259080\cdot2^{6n}. 
\end{cases}
    \end{equation}
 Combining \eqref{eq 3.12} and \eqref{eq 3.13} we compute $a_{i}^{(j)}$ in \eqref{eq 3.12}  for $7\leqslant j\leqslant 9 , \; 0 \leqslant i \leqslant  j-1 $\\
 and we obtain from  \eqref{eq 3.12}  \\
    \begin{equation}
  \label{eq 3.14}
  \Gamma_{i}^{\left[2\atop{\vdots \atop 2}\right]\times 9} = \begin{cases}
1 & \text{if  } i = 0,        \\
 (2^{n}-1)\cdot 3 & \text{if   } i=1,\\
 7\cdot2^{2n}+999 \cdot 2^{n}-1006  & \text{if   }  i = 2,  \\
 15\cdot 2^{3n}+3451\cdot 2^{2n}
-10458\cdot 2^{n}+6992  & \text{if   }  i = 3,  \\
 31\cdot 2^{4n} +8355\cdot 2^{3n}+111118 \cdot 2^{2n}
-392304 \cdot2^{n}+272800 & \text{if   }  i=4, \\ 
 63 \cdot 2^{5n}+17267 \cdot2^{4n} +356190 \cdot2^{3n}
 -3107440 \cdot 2^{2n}+6568032 \cdot 2^n -3834112 & \text{if   }  i=5, \\
 127\cdot 2^{6n}+31059\cdot2^{5n}+ 492094\cdot 2^{4n}  -6658800\cdot 2^{3n}\\
 +24491488 \cdot 2^{2n}-35215104\cdot2^{n} +16859136 & \text{if   }  i=6, \\
  255\cdot 2^{7n}  +42291\cdot2^{6n}  -219618\cdot 2^{5n}-4053808\cdot 2^{4n}\\ +
 32840160\cdot 2^{3n}-82168576 \cdot 2^{2n}+ 81543168\cdot2^{n} -27983872& \text{if   }  i=7, \\
   511\cdot2^{8n} -765\cdot 2^{7n}  -127762\cdot2^{6n}  +440496\cdot2^{5n}+ 8456800\cdot 2^{4n} \\ 
57511680 \cdot 2^{3n}+118013952\cdot 2^{2n}-83951616\cdot 2^{n} +14680064 & \text{if   }  i=8, \\
 2^{10n} -511\cdot 2^{8n}+ 510\cdot 2^{7n}  +85344\cdot 2^{6n}  -252000\cdot 2^{5n}-4912384\cdot 2^{4n} \\ +
 30965760\cdot 2^{3n}-57344000 \cdot 2^{2n}+ 31457280 \cdot 2^{n} & \text{if   }  i=9. \\
  \end{cases}    
  \end{equation}
 We recall the similar results obtained in Cherly [11,12] for $1\leqslant k\leqslant 8 $\\
 
  \textbf{The case k=1}
  \begin{equation*}
       \Gamma_{i}^{\left[2\atop{\vdots \atop 2}\right]\times 1} =  \begin{cases}
1 & \text{if  } i = 0,        \\
   2^{2n} -1& \text{if   } i=1.
 \end{cases}    
   \end{equation*}

\textbf{The case k=2}

  \begin{equation*}
       \Gamma_{i}^{\left[2\atop{\vdots \atop 2}\right]\times 2} =  \begin{cases}
1 & \text{if  } i = 0,        \\
   (2^{n}-1)\cdot 3 & \text{if   } i=1,\\
2^{  3n} - 3\cdot 2^{n} +2 & \text{if   }  i=2. 
    \end{cases}    
   \end{equation*}

\textbf{The case k=3}

    \begin{equation*}
       \Gamma_{i}^{\left[2\atop{\vdots \atop 2}\right]\times 3} =  \begin{cases}
1 & \text{if  } i = 0,        \\
   (2^{n}-1)\cdot 3 & \text{if   } i=1,\\
  7\cdot 2^{2n} -9\cdot2^{n}+2 & \text{if   }  i = 2,  \\
2^{  4n} - 7\cdot 2^{2n} +6\cdot2^{n}  & \text{if   }  i=3. 
    \end{cases}    
   \end{equation*}

\textbf{The case k=4}

       \begin{equation*}
      \Gamma_{i}^{\left[2\atop{\vdots \atop 2}\right]\times 4}=   \begin{cases}
1 & \text{if  } i = 0,        \\
 (2^{n}-1)\cdot 3 & \text{if   } i=1,\\
7\cdot2^{2n}+7\cdot2^{n}-14 & \text{if   }  i = 2,  \\
 15\cdot 2^{3n}-21\cdot 2^{2n}-42\cdot 2^{n}+48   & \text{if   }  i = 3,  \\
 2^{5n} -15\cdot 2^{3n}+7\cdot 2^{2n+1}+2^{n+5}-32 & \text{if   }  i=4. 
    \end{cases}    
   \end{equation*}

\textbf{The case k=5}
  \begin{equation*}
    \Gamma_{i}^{\left[2\atop{\vdots \atop 2}\right]\times 5}=   \begin{cases}
1 & \text{if  } i = 0,        \\
 (2^{n}-1)\cdot 3 & \text{if   } i=1,\\
7\cdot2^{2n}+39 \cdot 2^{n}-46  & \text{if   }  i = 2,  \\
 15\cdot 2^{3n}+91\cdot 2^{2n}-189\cdot 2^{n+1}+272   & \text{if   }  i = 3,  \\
31\cdot 2^{4n} -45\cdot 2^{3n}-161\cdot 2^{2n+1}+51\cdot2^{n+4}-480 & \text{if   }  i=4, \\                                                                                                
2^{6n} - 31\cdot 2^{4n} +15\cdot 2^{3n+1}+7\cdot 2^{2n+5}-15\cdot2^{n+5}+256   & \text{if   }  i=5. 
  \end{cases}    
 \end{equation*}

\textbf{The case k=6}

  \begin{equation*}
    \Gamma_{i}^{\left[2\atop{\vdots \atop 2}\right]\times 6}=   \begin{cases}
1 & \text{if  } i = 0,        \\
 (2^{n}-1)\cdot 3 & \text{if   } i=1,\\
7\cdot2^{2n}+103 \cdot 2^{n}-110  & \text{if   }  i = 2,  \\
 15\cdot 2^{3n}+315\cdot 2^{2n}-1050\cdot 2^{n}+720   & \text{if   }  i = 3,  \\
31\cdot 2^{4n} +515\cdot 2^{3n}-2450 \cdot 2^{2n}+3280 \cdot2^{n}-1376  & \text{if   }  i=4, \\                                                                                                
   63 \cdot 2^{5n}-93 \cdot2^{4n} -1650 \cdot2^{3n}+5040 \cdot 2^{2n}-4128 \cdot 2^n +768 & \text{if   }  i=5, \\
      2^{7n}-63 \cdot 2^{5n}+62\cdot 2^{4n}+1120\cdot 2^{3n}-2912 \cdot 2^{2n}+1792 \cdot2^{n}   & \text{if   }  i=6. 
  \end{cases}    
  \end{equation*}

  \textbf{The case k=7}

   \begin{equation*}
   \Gamma_{i}^{\left[2\atop{\vdots \atop 2}\right]\times 7}=  \begin{cases}
1 & \text{if  } i = 0,        \\
 (2^{n}-1)\cdot 3 & \text{if   } i=1,\\
7\cdot2^{2n}+231 \cdot 2^{n}-238  & \text{if   }  i = 2,  \\
15\cdot 2^{3n}+763\cdot 2^{2n}-2394\cdot 2^{n}+1616   & \text{if   }  i = 3,  \\
 31\cdot 2^{4n} +1635\cdot 2^{3n}-2610 \cdot 2^{2n}\\
 -4080 \cdot2^{n}+5024  & \text{if   }  i=4, \\  
63 \cdot 2^{5n}+2387 \cdot2^{4n} -11970 \cdot2^{3n}\\
-9520 \cdot 2^{2n}+74592 \cdot 2^n -55552 & \text{if   }  i=5, \\
    127\cdot 2^{6n}-189\cdot2^{5n} -7378\cdot 2^{4n}\\
    +24240\cdot 2^{3n}+35168 \cdot 2^{2n}-166656 \cdot2^{n} +114688 & \text{if   }  i=6, \\
   2^{8n}- 127\cdot2^{6n}+126\cdot 2^{5n}+4960\cdot 2^{4n}\\
   -13920\cdot 2^{3n}-23808 \cdot 2^{2n}+98304 \cdot2^{n}  -65536 & \text{if   }  i=7.    
  \end{cases}    
  \end{equation*}
  
\textbf{The case k=8}
  \begin{equation*}
  \label{eq 6.14}
  \Gamma_{i}^{\left[2\atop{\vdots \atop 2}\right]\times 8} = \begin{cases}
1 & \text{if  } i = 0,        \\
 (2^{n}-1)\cdot 3 & \text{if   } i=1,\\
 7\cdot2^{2n}+487 \cdot 2^{n}-494  & \text{if   }  i = 2,  \\
15\cdot 2^{3n}+1659\cdot 2^{2n}\\
-5082\cdot 2^{n}+3408   & \text{if   }  i = 3,  \\
 31\cdot 2^{4n} +3875\cdot 2^{3n}+13454 \cdot 2^{2n}\\
-67952 \cdot2^{n}+50592  & \text{if   }  i=4, \\  
 63 \cdot 2^{5n}+7347 \cdot2^{4n} +28830 \cdot2^{3n}\\
 -468720 \cdot 2^{2n}+1092192 \cdot 2^n -659712 & \text{if   }  i=5, \\
127\cdot 2^{6n}+10227\cdot2^{5n} -52514\cdot 2^{4n}\\
-339760\cdot 2^{3n}+2548448 \cdot 2^{2n}-4804352 \cdot2^{n} +2637824& \text{if   }  i=6, \\
    255\cdot 2^{7n}-381\cdot2^{6n}-31122\cdot2^{5n} \\
    +105648\cdot 2^{4n}+758880\cdot 2^{3n}-4617984 \cdot 2^{2n}+7913472 \cdot2^{n} -4128768 & \text{if   }  i=7, \\
2^{9n}- 255\cdot 2^{7n} +254\cdot2^{6n}+20832\cdot 2^{5n}\\
-60512\cdot 2^{4n}-451840\cdot 2^{3n}+2523136 \cdot 2^{2n}-4128768 \cdot2^{n}  +2097152 & \text{if   }  i=8.    
  \end{cases}    
  \end{equation*}

  \begin{example}
  \textbf{Computation of  $ R_{q,n}^{(k)} $ in the case k=9, q=3\; (see \eqref{eq 3.2} and \eqref{eq 3.3}) }
     The number denoted by $ R_{3,n}^{(9)} $ of solutions \\
  
 $(Y_1,U_{1}^{(1)},U_{2}^{(1)}, \ldots,U_{n}^{(1)}, Y_2,U_{1}^{(2)},U_{2}^{(2)}, 
\ldots,U_{n}^{(2)}, Y_3,U_{1}^{(3)},U_{2}^{(3)}, \ldots,U_{n}^{(3)}   ) \in (\mathbb{F}_{2}[T])^{3(n+1)}$ \vspace{0.5 cm}\\
 of the polynomial equations  \vspace{0.5 cm}
  \[\left\{\begin{array}{c}
 Y_{1}U_{1}^{(1)} + Y_{2}U_{1}^{(2)} + Y_{3}U_{1}^{(3)} = 0  \\
    Y_{1}U_{2}^{(1)} + Y_{2}U_{2}^{(2)}  + Y_{3}U_{2}^{(3)} = 0\\
    \vdots \\
   Y_{1}U_{n}^{(1)} + Y_{2}U_{n}^{(2)} + Y_{3}U_{n}^{(3)} = 0 
 \end{array}\right.\]
 
  satisfying the degree conditions \\
                   $$  degY_i \leq 8,
                   \quad degU_{j}^{(i)} \leq 1, \quad  for \quad 1\leq j\leq n,   \quad 1\leq i \leq 3 $$ \\
  is equal to

  \begin{align*}
 R_{q,n}^{(k)}  =  2^{q(2n+k) - (k+1)n}\sum_{i = 0}^{\inf(2n,k)}
  \Gamma_{i}^{\left[2\atop{\vdots \atop 2}\right]\times k} 2^{-iq} = R_{3,n}^{(9)} =  
   2^{27-4n}\sum_{i = 0}^{9} \Gamma_{i}^{\left[2\atop{\vdots \atop 2}\right]\times 9} 2^{-i3}\\
    = 2^{27-4n}\cdot \bigg[ 2^{-27}\cdot 2^{10n} \\+(-511\cdot 2^{-27}+511\cdot 2^{-24})\cdot2^{8n}\\
  +(510\cdot 2^{-27}-765\cdot 2^{-24}+255\cdot 2^{-21})\cdot2^{7n}\\
  +(85344\cdot2^{-27}-127762\cdot 2^{-24}  + 42291\cdot 2^{-21}+127\cdot 2^{-18})\cdot2^{6n}\\
  +(-252000\cdot2^{-27}+440496\cdot2^{-24}  -219618\cdot2^{-21}+31059\cdot2^{-18}+63\cdot2^{-15})\cdot2^{5n}\\ 
 +(-4912384\cdot2^{-27}+8456800\cdot2^{-24}  -4053808\cdot2^{-21}\\
 +492094\cdot2^{-18}+17267\cdot2^{-15}+31\cdot2^{-12})\cdot2^{4n}\bigg] \\
 = 2^{6n}+3577\cdot 2^{4n}+10710\cdot 2^{3n}+1834896 \cdot 2^{2n}+5376672\cdot 2^{n}+126991872.
  \end{align*}
 \textbf{ The case n=1:}\\
 
  $  R_{3,1}^{(9)}= 2^{6}+3577\cdot 2^{4}+10710\cdot 2^{3}+1834896 \cdot 2^{2}+5376672\cdot 2+126991872 =145227776. $
  
  \textbf{ The case n=2:}\\
 
  $  R_{3,2}^{(9)}= 2^{12}+3577\cdot 2^{8}+10710\cdot 2^{6}+1834896 \cdot 2^{4}+5376672\cdot 2^2+126991872 =179462144. $

  \textbf{ The case n=3:}\\
 
  $  R_{3,3}^{(9)}= 2^{18}+3577\cdot 2^{12}+10710\cdot 2^{9}+1834896 \cdot 2^{6}+5376672\cdot 2^3+126991872 =307835648. $
    \end{example}


\newpage 
      \begin{thebibliography}{99}
\bibitem{Landsberg}Landsberg, G {Ueber eine Anzahlbestimmung und eine damit zusammenhangende Reihe},
 {J. reine angew.Math}, {\bf 111}(1893),87-88. 
 \bibitem{Fisher and Alexander}Fisher,S.D and Alexander M.N. {Matrices over a finite field}\\
 {Amer.Math.Monthly 73}(1966), 639-641
  \bibitem{Daykin}  Daykin David E,  {Distribution of Bordered Persymmetric Matrices in a finite field}
  {J. reine angew. Math}, {\bf 203}(1960) ,47-54
   \bibitem{Cherly} Cherly, Jorgen. \\ {Exponential sums and rank of persymmetric  matrices over  $\mathbf{F}_2 $  }\\
{arXiv : 0711.1306, 46 pp} 
  \bibitem{Cherly} Cherly, Jorgen. \\ {Exponential sums and rank of  double  persymmetric  matrices over  $\mathbf{F}_2 $  }\\
{arXiv : 0711.1937, 160 pp} 
  \bibitem{Cherly} Cherly, Jorgen. \\ {Exponential sums and rank of  triple  persymmetric  matrices over  $\mathbf{F}_2 $  }\\
{arXiv : 0803.1398, 233 pp} 
  \bibitem{Cherly} Cherly, Jorgen. \\ {Results about   persymmetric  matrices over  $\mathbf{F}_2 $ and related exponentials sums }\\
   {arXiv : 0803.2412v2, 32 pp} 
   \bibitem{Cherly} Cherly, Jorgen. \\ {Polynomial equations and rank of  matrices over  $\mathbf{F}_2 $ related  to persymmetric matrices}\\
   {arXiv : 0909.0438v1, 33 pp}   
      \bibitem{Cherly} Cherly, Jorgen. \\ {On a conjecture regarding enumeration of n-times persymmetric matrices over $\mathbf{F}_2 $ by rank}\\
   {arXiv : 0909.4030, 21 pp}   
      \bibitem{Cherly} Cherly, Jorgen. \\ { On a conjecture concerning the fraction  of  invertible  m-times  Persymmetric  Matrices over $\mathbb{F}_{2} $}\\ 
    {arXiv : 1008.4048v1, 11 pp}  
    \bibitem{Cherly} Cherly, Jorgen. \\ {Enumeration of some particular n-times  Persymmetric  Matrices over $\mathbb{F}_{2} $ by rank}\\ 
    {arXiv : 1101.2097v1, 18 pp}  
       \bibitem{Cherly} Cherly, Jorgen. \\ {Enumeration of some particular quadruple Persymmetric  Matrices over $\mathbb{F}_{2} $ by rank}\\ 
    {arXiv : 1106.2691v1, 21 pp}  
       \bibitem{Cherly} Cherly, Jorgen. \\ {Enumeration of some particular quintuple Persymmetric  Matrices over $\mathbb{F}_{2} $ by rank}\\ 
    {arXiv : 1109.3623v1, 23 pp}      
     \end{thebibliography}
 \end{document}